\documentclass[a4paper,12pt]{article}

\usepackage{xcolor}

\newcommand{\dd}{\mathrm{d}} 

\usepackage[margin=3cm,footskip=1cm]{geometry}

\usepackage[T1]{fontenc} 
\usepackage{amsmath}

\usepackage{amsmath,amssymb,amsthm,bm,mathtools,xspace,booktabs,url}
\usepackage{graphicx,etoolbox}

\usepackage[shortlabels]{enumitem}
\setlist{
  listparindent=\parindent,
  parsep=0pt,
}

\usepackage[raggedright]{titlesec}

\usepackage{appendix}
\numberwithin{equation}{section} 

\theoremstyle{plain} 
\newtheorem{theorem}{Theorem}[section]
\newtheorem{lemma}[theorem]{Lemma}
\newtheorem{proposition}[theorem]{Proposition}

\theoremstyle{definition} 

\usepackage{authblk}

\makeatletter
\newcommand\CorrespondingAuthor[1]{%
  \begingroup%
  \def\@makefnmark{}%
  \footnotetext{Corresponding author: #1}%
  \endgroup%
}
\makeatother

\makeatletter
\renewenvironment{abstract}{%
  \small%
  \begin{center}%
    \bfseries \abstractname\vspace{-.5em}\vspace{\z@}%
  \end{center}%
  \quote%
}{\endquote}
\makeatother

\usepackage[above,below,section]{placeins}

\usepackage{subfloat}

\usepackage{siunitx}

\usepackage{lipsum}

\usepackage[utf8]{inputenc}
\usepackage{float,multicol,eso-pic,rotating,bbm}

\usepackage{subfig}
\makeatletter
\DeclareRobustCommand*\subref{\@ifstar\sf@@subref\sf@subref}
\makeatother

\font\dsrom=dsrom10 scaled 1200

\newcommand{\R}{\mathbb{R}}

\newcommand{\bx}{\mathbf{x}}
\newcommand{\bX}{\mathbf{X}}
\newcommand{\LnR}{{\Lambda_{n,\widetilde R}}}

\DeclareMathOperator\Var{Var}

\DeclareMathOperator\E{E}

\newcommand{\ind}{\textrm{\dsrom{1}}}

\let\oldtextbf\textbf
\renewcommand\textbf[1]{\oldtextbf{\boldmath #1}}

\title{\ Poisson intensity parameter estimation for stationary Gibbs point processes of finite interaction range}
\author{Jean-Fran\c{c}ois Coeurjolly and Nadia Morsli

Jean-Francois.Coeurjolly@upmf-grenoble.fr; nadia.morsli@imag.fr
Laboratory Jean Kuntzmann, Grenoble University, France.}
\date{}

\begin{document}
\maketitle

\begin{abstract}
We introduce a semi-parametric estimator of the Poisson intensity parameter of a  spatial stationary Gibbs point process. Under very mild assumptions satisfied by a large class of Gibbs models, we establish its strong consistency and asymptotic normality. We also consider its finite-sample properties in a simulation study.
\end{abstract}
 
\section{Introduction}

Spatial Gibbs  point processes form a major class of stochastic processes allowing the modelling  of dependence of spatial point patterns. They have applications in many scientific fields such as biology, epidemiology, geography, astrophysics, physics, and economics.  General references covering as well theoretical as practical aspects of these processes are e.g. \cite{B-StoKenMecRus87}, \cite{B-MolWaa04} or \cite{B-IllPenSto08}.

Gibbs point processes in $\R^d$ can be defined and characterized through the Papangelou conditional intensity (see \cite{B-MolWaa04}) which is a function $\lambda:\R^d \times \Omega \to \R_+$ where $\Omega$ is the space of locally finite configurations of points in $\R^d$. The  Papangelou conditional intensity can be interpreted as follows: for any $u\in \R^d$ and $\bx \in \Omega$, $\lambda(u,\bx)\dd u$ corresponds to the conditional probability of observing a point in a ball of volume $\dd u$ around $u$ given the rest of the point process is $\bx$. In the present paper, we assume that the Papangelou conditional intensity is decomposed as
\begin{equation}\label{eq:lbeta}
  \lambda(u,\bx) = \beta \, \widetilde{\lambda}(u,\bx),
\end{equation}
where $\beta$ is a positive real parameter and where $\widetilde \lambda:\R^d \times \Omega\to \R_+$. If the point process corresponds to a homogeneous Poisson point process $\widetilde\lambda(u,\bx)=1$. We further assume that $\widetilde \lambda(u,\emptyset)=1$ which in case means that $\beta=\lambda(u,\emptyset)$ and that $\widetilde \lambda$ represents the higher order interaction term. Therefore, we suggest the name {\it Poisson intensity parameter} for the parameter $\beta$. The goal of this paper is to develop an estimator of $\beta$ without specifying anything on the function $\widetilde \lambda$. Our semi-parametric estimator is based on a smart application of the Georgii-Nguyen-Zessin formula (see \cite{A-Geo76} for a general presentation and Section \ref{sub}) which is also a way to characterize a Gibbs point process.
We propose a very simple ratio-estimator of $\beta$ based on a single observation of a stationary spatial Gibbs point process  observed in a bounded window and we prove its asymptotic properties (strong consistency and asymptotic normality) as the window expands to $\R^d$. 
The problem considered in this paper is related to two different works that we detail below. 

First, due to its easy interpretability the most popular Gibbs model is without any doubt the (isotropic) pairwise  interaction point process. Its Papangelou conditional intensity is given by
\begin{equation}\label{eq:pair}
   \lambda(u,\bx) = \beta \widetilde \lambda(u,\bx) \quad \mbox{ with } \quad
    \widetilde \lambda (u,\bx) = \prod_{v \in \bx\setminus u} g(\| v-u\|).
 \end{equation} 
The function $g$ is called the (isotropic) pairwise interaction function. Estimating the model~\eqref{eq:pair} consists in estimating $\beta$ and the function $g$. To the best of our knowledge, this awkward problem has been considered only in \cite{A-DigGatSti87}. The authors propose several methods to estimate non parametrically the function $g$. Our objective is to make some advances in that direction which justifies first to propose an estimator of $\beta$ independently of the function $g$ and to understand its properties.

Second, many popular methods exist to estimate a parametric Gibbs model. This includes methods based on the likelihood or the pseudo-likelihood and the Takacs-Fiksel method. We refer for instance to \cite{B-MolWaa04}, \cite{B-IllPenSto08} or \cite{C-Mol10} and the references therein for an overview of these methods. A recent contribution \cite{A-BadDer12} investigates a new approach. The following exponential model is considered
\begin{equation}\label{eq:exp}
   \lambda(u,\bx) = \beta \widetilde \lambda(u,\bx) \quad \mbox{ with } \quad
   \log \widetilde \lambda (u,\bx) = \theta^\top v(u,\bx)
 \end{equation} 
where $\theta$ is a real $p$-dimensional parameter vector, $v(u,\bx)=(v_1(u,\bx),\ldots,v_p(u,\bx))^\top$ for measurable functions $v_i:\R^d \times \Omega \to \R$, $i=1,\dots,p$. \cite{A-BadDer12} develops a new estimator of the parameter $\theta$ where the parameter $\beta$ is treated as a nuisance parameter. The main advantage of the estimator is that it is obtained via the resolution of a linear system of  equations whereas all previous methods cited above require an optimization procedure. The ratio-estimator of $\beta$ we propose is also very  computationally cheap. Hence, it could  serve as a complement of the estimate of $\theta$ proposed in~\cite{A-BadDer12} to completely identify the model~\eqref{eq:exp}.

The rest of the paper is organized as follows. Section~\ref{sec:background} gives the background and notation, presents examples of Gibbs point processes and our main assumptions. Section~\ref{sec:estimator} deals with the core of the paper. We present our estimator and derive its asymptotic properties as the window of observation expands to $\R^d$. A simulation study is conducted in Section~\ref{sec:simuls} where we illustrate the efficiency of our estimator. Proofs of the results are postponed to Appendix. 

\section{Background and assumptions}\label{sec:background}

\subsection{Basic definitions and tools}\label{sub}
 
A point process $\bX$ in $\R^d$ is a locally
finite random subset of $\R^d$, i.e. the number of points  $N (\Lambda)= n(\bX_\Lambda)$ of the restriction of $\bX$ to $\Lambda$ is a 
finite random variable whenever $\Lambda$ is a bounded Borel set of $\R^d$ (see \cite{B-DalVer88}). 
We let $\Omega$ be the space of locally finite point configurations. If the distribution of $\bX$ is translation invariant, we say that $\bX$ is stationary. 

The Papangelou conditional intensity completely characterizes the Gibbs point process in terms of the Georgii-Nguyen-Zessin (GNZ) Formula (see  \cite{Papangelou} and
\cite{A-Zes09} for historical comments and \cite{A-Geo76} or
\cite{A-NguZes79b} for a general presentation). The GNZ formula states that for any measurable function 
$h:\R^d\times \Omega$ such that the left or right hand side exists
\begin{equation}\label{eq:gnz}
\E\displaystyle \sum_{u\in{\bX} }h(u, \bX\setminus u)= \E\displaystyle \int_{\R^d} h(u,\bX)
\lambda(u,\bX)\dd u. 
\end{equation}
We will not discuss how to consistently specify the Papangelou conditional intensity to ensure the existence of a Gibbs point process on $\R^d$, but rather we simply assume we are given a well-defined Gibbs point process. The reader interested in a deeper presentation of Gibbs point processes and the existence problem is referred to \cite{B-Rue69,B-Pre76} or
\cite{A-DerDroGeo12}, see also Section~\ref{sec:examples} for a few examples.

The concept of innovation for spatial point processes is proposed in \cite{A-BadTurMolHaz05} and is a key-ingredient in this paper. Inspired by the GNZ formula, it is defined as follows: for some measurable function $h:\R^d \times \Omega\to \R$, the $h-$weighted innovation
 computed in a bounded domain $\Lambda$
is the centered random variable defined by
\begin{equation}\label{eq:defInn}
I_{\Lambda} (\bX,h)= \displaystyle \sum_{u\in{\bX_{\Lambda}} }h(u, \bX\setminus u)-\displaystyle \int_{\Lambda} h(u,\bX)\lambda(u,\bX) \dd u.
\end{equation}

We end this paragraph with a few notation. 
The volume of a bounded Borel set $\Lambda$ of $\R^d$ is denoted by $|\Lambda|$.
The norm $\|\cdot\|$ stands for the standard Euclidean norm. For two subsets $\Lambda_1$ and $\Lambda_2$ of $\R^d$, $d(\Lambda_1,\Lambda_2)$ corresponds to the minimal distance between these two subsets. Finally, $ B(u,R)$ denotes the closed ball centered at $u\in \R^d$ with radius $R$.

\subsection{Examples of Gibbs point processes} \label{sec:examples}

We present in this paragraph several classical examples through their Papangelou conditional intensity. Further details on these models can be found in \cite{B-MolWaa04} for examples (i)-(iv) and (vi)-(vii) and to \cite{A-Gey99} for the example (v). 
Let $u\in \R^d, \bx\in \Omega$ and $R>0.$

(i) Strauss point process. 
\[
 \lambda(u,\bx)=\beta\gamma^{ n_{[0,R]}(u,\bx \setminus u)} 
 \]
where $\beta> 0, \gamma \in [0,1]$ and  $n_{[0,R]}(u,\bx)= \displaystyle\sum_{v\in \bx} \ind(\|v-u\|\leq R)$
represents the number of $R$-closed neighbours of $u$ in $\bx.$

(ii) Strauss point process with hard-core. Let $0<\delta<R$ the Papangelou conditional intensity is 
\[
  \lambda(u,\bx) = \left\{ 
  \begin{array}{ll}
    0 & \mbox{ if } d(u,\bx \setminus u)\leq \delta\\
    \beta\gamma^{ n_{]\delta,R]}(u,\bx \setminus u)} & \mbox{ otherwise}
   \end{array} 
  \right.
\]
where $\beta> 0, \gamma \in [0,1]$.

(iii) Piecewise Strauss point process. 
\[
  \lambda(u,\bx)=\beta\prod_{j=1}^p  \gamma_j ^{n_{[R_{j-1},R_j]}(u,\bx \setminus u)}
\]
 where $\beta>0, \gamma_j \in [0,1], n_{[R_{j-1},R_j]}(u,\bx)=\displaystyle\sum_{v\in \bx} \ind(\|v-u\|\in [R_{j-1},R_j])$ and $R_0=0<R_1<\ldots<R_p=R<\infty$.
 
(iv) Triplets point process.  
\[
 \lambda(u,\bx)=\beta\gamma^{s_{[0,R]}(\bx\cup u)-s_{[0,R]}(\bx\setminus u)}
\]
 where $\beta>0, \gamma\in [0,1]$ and $s_{[0,R]}(\bx)$ is the number of unordered triplets that are closer than $R$. 

(v) Geyer saturation point process  with saturation threshold $s$
\[
   \lambda(u,\bx)= \beta \gamma^{t(\bx\cup u) -t(\bx \setminus u)}
\]
where   $\beta > 0, \gamma \in [0,1], s\geq 1$ and $t(\bx)= \sum_{v\in \bx}\min(s,n_{[0,R/2]}(v,\bx\setminus v))$.

(vi) Lennard-Jones model with finite range. This corresponds to the pairwise interaction point process defined by
\[
 \lambda(u,\bx)= \beta \prod_{v\in \bx \setminus u} g(\| v-u\|)\big )
 \]
with for $r>0$, $\log g(r)=  \big( \theta^6 r^{-6}- \theta^{12}r^{-12}  \big)\ind_{]0,R]}(r), $
$\theta>0$ and $\beta>0$.

(vii) Area-interaction point process.
\[
\lambda(u,\bx)= \beta \gamma^{A(\bx\cup u)-A(\bx \setminus u)}
\]
for $\beta>0, \gamma>0$ and where $A(\bx)=|  \cup_{v\in \bx} B(v,R/2)|.$

Examples (i)-(iii) and (v)-(vi) are pairwise interaction point processes whereas
example (iv) (resp. (vii)) is based on interactions on cliques of order 3 (resp. of any order).  Example (vi) is the only example which is not locally stable but only Ruelle superstable, see  \cite{B-Rue69}. We recall that the local stability property states that (for stationary Gibbs models) the Papangelou conditional intensity is uniformly bounded by a constant (see \cite{B-MolWaa04}).

\subsection{Main assumptions and discussion}

In this section, we consider the model~\eqref{eq:lbeta} and we present the assumptions we require to propose in Section~\ref{sec:defEst} an estimator of the { Poisson intensity parameter}. To stress on the fact that the Papangelou conditional intensity satisfies~\eqref{eq:lbeta} and that we aim at estimating $\beta$ we denote, from now on, the corresponding Papangelou conditional intensity $\lambda_\beta$.

The two following assumptions will be considered and discussed throughout the paper.

(a) The  Papangelou conditional intensity  has a finite range $R$, i.e. 
\begin{equation}\label{eq:range}
 \lambda_{\beta}(u,\bx)= \lambda_{\beta}(u,\bx_{B(u, R)}),
\end{equation}
for any  $u\in \R^d$, $ \bx\in \Omega, \beta>0$. 

(b) The function $\widetilde\lambda$ satisfies for any $u\in \R^d$
\begin{equation}\label{eq:ltilde}
\widetilde{\lambda}(u,\emptyset) = 1. 
\end{equation}

As this paper deals with stationary Gibbs point processes, we stress on the fact tha the Papangelou conditional intensity is invariant by translation for any $\beta>0$. We emphasize that no more assumption will be required.
In particular we do not require that the Gibbs point process is locally stable. This is noticeable since local stability is often an assumption  
made when dealing with such 
models, see e.g. \cite{A-BilCoeDro08}. We also underline that in the following, 
we do not assume that $R$ is known. In place of this, we assume we are given an upper-bound of $R$.

Let us discuss these assumptions regarding the examples presented in Section~\ref{sec:examples}.
Examples (i)-(vii) exhibit Papangelou conditional intensities 
which satisfy \eqref{eq:lbeta} and \eqref{eq:range} with finite 
range $R$.   Example (i)-(vi) 
satisfy the assumption~\eqref{eq:ltilde}, whereas example (vii) is a counter-example since for the latter one
$\widetilde{\lambda}(u,\emptyset)=\gamma^{\pi(R\diagup2)^2} \neq 1$ 
for $\gamma>0$. Condition (b) can thus be viewed as an identifiability condition.

\section{Semi-parametric estimator of $\beta^{\star}$} \label{sec:estimator}

\subsection{Definition} \label{sec:defEst}

In the following, we let $\beta^\star>0$ denote the true parameter and we denote by  
$P_{\beta^\star}$ the probability measure of $\bX$, i.e. $\bX \sim P_{\beta^\star}$.

Our estimator is based on an application of the GNZ formula \eqref{eq:gnz} for the following function $h$ given for all $\  u\in \R^d$, $\bx\in \Omega$ 
 and $\widetilde{R}\geq R$ by 
\begin{equation} \label{eq:defh}
h(u,\bx)=\ind(d(u,\bx)>\widetilde R)=\ind(\bx\cap  B(u,\widetilde{R})=\emptyset). 
\end{equation}
 The idea is formulated by the following result. Let $\Lambda$ be a bounded region of $\R^d$ and let
 $N_{\Lambda}(\bX;\widetilde{R})$ and $V_{\Lambda}(\bX;\widetilde{R})$ denote the following random variables
\begin{align*}
N_\Lambda(\bX;\widetilde R) &= \sum_{u\in{\bX_{\Lambda}} }\ind ((\bX \setminus u)\cap  B(u,\widetilde{R})=\emptyset)\\
V_\Lambda(\bX;\widetilde R) &= \int_\Lambda\ind (\bX \cap  B(u,\widetilde{R})=\emptyset) \dd u.\\
\end{align*}
The variable $N_\Lambda$ corresponds to the number of points in $X_\Lambda$ that are separated from other points in $X$ by a distance greater than $\widetilde R$ while the variable $V_\Lambda$ is the volume of $\Lambda$ after digging a disk hole of radius $\widetilde R$ centered at each point of $X$.


\begin{proposition}\label{prop:h} Assume~\eqref{eq:lbeta},~\eqref{eq:range} and~\eqref{eq:ltilde},
then we have for any $\widetilde R\geq R$
\begin{equation}
\E N_{\Lambda}(\bX;\widetilde{R}) =\beta^\star\E V_{\Lambda}(\bX;\widetilde{R})
\end{equation}
 where $\E$ denotes the expectation with respect to $P_{\beta^\star}$.
\end{proposition}

\begin{proof}
Using the GNZ formula \eqref{eq:gnz} applied to the function $h$ defined by \eqref{eq:defh}, we obtain
\begin{align*}
\E N_{\Lambda}(\bX;\widetilde{R}) &= \E \sum_{u\in{\bX_{\Lambda}} }\ind ((\bX \setminus u)\cap  B(u,\widetilde{R})=\emptyset) \\
&= \beta^\star \E \int_{\Lambda}\ind(\bX\cap  B(u,\widetilde{R})=\emptyset )\widetilde{\lambda}(u,\bX) \dd u .
\end{align*}
Now, using the finite range property \eqref{eq:range}, the assumption \eqref{eq:ltilde} and the fact that $\widetilde R\geq R$ we continue with
\begin{align*}
\E N_{\Lambda}(\bX;\widetilde{R}) &=\beta^\star \E \int_{\Lambda}\ind(\bX\cap  B(u,\widetilde{R})=\emptyset ) \widetilde{\lambda}(u,\bX\cap  B(u,R)) \dd u \\
&= \beta^\star \E \int_\Lambda \ind(\bX\cap  B(u,\widetilde{R})=\emptyset ) \widetilde\lambda(u,\emptyset) \dd u \\
&= \beta^\star\E V_{\Lambda}(\bX;\widetilde{R}).
\end{align*}
\end{proof}

Using a classical ergodic theorem for spatial point processes obtained in \cite{A-NguZes79}, we can expect that when the observation domain is large 
$N_\Lambda(\bX;\widetilde R)\simeq \E N_{\Lambda}(\bX;\widetilde R)$ and $V_\Lambda(\bX;\widetilde R)\simeq \E V_{\Lambda}(\bX;\widetilde R)$. Using this and Proposition~\ref{prop:h}, we are suggested a natural and computationally cheap ratio-estimator of the parameter $\beta^\star$. We define this estimator and present its asymptotic properties in the next section.

\subsection{Asymptotic properties}

To derive asymptotic properties of our procedure when the window expands to $\R^d$, we assume that  $\bX$ is observed in 
$\Lambda_n$ where $(\Lambda_n)_{n\geq 1}$ is a sequence of cubes growing up to $\R^d$. 

When $u$ is a point close to the boundary of the window
$\Lambda_n$, we can only evaluate $h(u,\bX \cap\Lambda_n\cap B(u, \widetilde R))$.
One simple strategy for eliminating the edge effect bias is the border method, i.e. by considering the erosion  
of $\Lambda_n$ by $\widetilde R$ defined by
\[
\LnR={\Lambda_n}_{\ominus \widetilde R}=\{u\in \Lambda_n: B(u,\widetilde R)\subseteq \Lambda_n \}.   
\]
Then we define our estimator by
\begin{equation}\label{eq:defEst}
\widehat{\beta}_n=\widehat{\beta}_n(\bX;\widetilde R)=
\frac{ N_\LnR(\bX;\widetilde{R})}{V_\LnR (\bX;\widetilde{R})}.
\end{equation}

We let $F$ denote the empty space function (see \cite{B-MolWaa04}) given, in the stationary case, by
\[
  F(r)=P_{\beta^\star}(\bX\cap B(0, r)\neq \emptyset) = 1 - P_{\beta^\star}(\bX\cap B(0, r)= \emptyset)  
\]
for $r>0$. 
Following this definition, by $F_{u,v}:\R_+\to [0,1]$ for any $u,v \in \R^d$ we denote the function given by
 \[
  F_{u,v}(r)=1 - P_{\beta^\star} \big( \bX\cap B(u, r) =  \emptyset,  \bX\cap B(v, r) = \emptyset \big)
 \]
 for $r>0$. We now state our main result. 

\begin{proposition} \label{prop:asymp}
Assume~\eqref{eq:lbeta},~\eqref{eq:range} and \eqref{eq:ltilde},
and let $\widetilde R\geq R$. Then we have the following statements as $n\to \infty$.

\noindent (i) $\widehat{\beta}_n$ is a strongly consistent estimator of $\beta^\star$.

\noindent (ii) If $P_{\beta^\star}$ is ergodic, the following convergence in distribution holds
\begin{equation} \label{eq:clt1}
|\LnR|^{1/2} (\widehat{\beta}_n-\beta^{\star})\stackrel{d}\longrightarrow \mathcal{N}
(0,\sigma^2)  
\end{equation}
where
\begin{equation}\label{eq:sigma}
\sigma^2=\sigma^2(\beta^{\star},\widetilde R)= \frac{\beta^\star}{1-F(\widetilde R)} + 
\frac{{\beta^\star}^2}{(1-F(\widetilde{R}))^2}  \int_{B(0, \widetilde R)}(1-F_{0,v}(\widetilde R))\dd v.
\end{equation}

\noindent (iii) The following convergence in distribution holds
\begin{equation} \label{eq:clt2}
 |\LnR|^{1/2}   \; \frac{(\widehat{\beta}_n-\beta^{\star})}{\widehat\sigma_n}
  \stackrel{d}\longrightarrow \mathcal{N}(0,1) 
\end{equation}
 where 
\begin{equation} \label{eq:sigmaEst}
 \widehat{\sigma}_n^2= \widehat\sigma_n^2(\bX;\widetilde R)= |\LnR|\left( 
 \frac{\widehat{\beta}_n}{V_\LnR(\bX;\widetilde{R}) } + 
 \frac{ \widehat{\beta}_n^2 \,W_\LnR(\bX;\widetilde{R}) }
{V^2_\LnR(\bX;\widetilde{R})}
\right)
\end{equation}
 is a strongly consistent estimator of $\sigma^2$ where
 \[
W_\LnR(\bX;\widetilde{R}) =
 \int_\LnR  \int_{B(u,\widetilde{R})\cap \LnR}\ind (\bX\cap B(u,\widetilde R)= \emptyset ) \ind(\bX \cap B(v,\widetilde{R})= \emptyset) \dd u \dd v.
\]
\end{proposition}

If the measure $P_{\beta^\star}$ is not ergodic, then it can be represented as a mixture of ergodic measures 
(see \cite[Theorem 14.10]{A-Geo77}). In this case, the asymptotic distribution in~\eqref{eq:clt1} becomes a mixture of Gaussian
 distributions. Since a mixture of standard Gaussian distribution is a standard Gaussian distribution, 
this explains why $P_{\beta^\star}$ is assumed to be ergodic for the result (ii) and omitted for the result (iii).

We underline the interesting and unexpected fact that the asymptotic variance $\sigma^2(\beta^\star,\widetilde R)$ (and its estimate) are independent on $\widetilde \lambda$ which is the unspecified part of the Papangelou conditional intensity.

 \section{Simulation study} \label{sec:simuls}
 
 In this section, we investigate the efficiency of the Poisson intensity estimate in a simulation study.
 We consider the models (i)-(v) described in Section~\ref{sec:examples}. We set $R=0.05$ and $\beta^\star=200$ 
for all the examples except for the model \texttt{g2} where we set $\beta^\star=50$. We consider
\begin{itemize}
  \item Strauss point process where models \texttt{s1} and \texttt{s2} have respectively $\gamma=0.2$ and $\gamma=0.8$.
  \item Strauss point process with hard-core with $\delta=R/2$  and where models \texttt{shc1} and \texttt{shc2} have respectively $\gamma=0.2$ and $\gamma=0.8$.
  \item Piecewise Strauss point process with $p=3$ and $(R_1,R_2,R_3)=(R/3,2R/3,R)$ and where models \texttt{ps1} and \texttt{ps2} have respectively $(\gamma_1,\gamma_2,\gamma_3)=(0.8,0.5,0.2)$ and $(\gamma_1,\gamma_2,\gamma_3)=(0.2,0.8,0.2)$.
  \item Triplets point process where models \texttt{t1} and \texttt{t2} have respectively ${\gamma}=0.2$ and $\gamma=0.8$.
  \item Geyer saturation point process with saturation threshold $s=1$  where models \texttt{g1} and \texttt{g2} have respectively ${\gamma}=0.5$ and $\gamma=1.5$.
\end{itemize}
 
 The estimate of the Poisson intensity we propose depends 
 on a parameter $\widetilde R$ and ought to work for any value larger than $R$. To illustrate this,  $\beta^\star$ is estimated for different values of $\widetilde R$. We consider
 $\widetilde{R}=R \times p$, where $p$ is a parameter varying from 0.9 to 1.2. 
 Each replication has been generated on the domain $[0,L]^2$ and the estimates are computed on $[\widetilde R,L-\widetilde R]^2$. To illustrate the convergence as the window grows up, the simulations are done with $L=1$ and $L=2$. The simulations are done using the \texttt{R} package \texttt{spatstat} (\cite{A-BadTur05}). The empirical results based on 500 replications are presented in Table~\ref{tab:sim1}. 

 The results are very satisfactory. We observe that for values of $\widetilde{R}\geq R$ (i.e. for values equal 
or larger than the finite range), 
 the estimates are almost unbiased whereas a non negligible bias is observed when $\widetilde{R}<R$. The bias is
  negative when the point process is repulsive (i.e. when $\lambda(u,\bx)\geq 1$) and positive when the point process is attractive (i.e. when $\lambda(u,\bx)\leq 1$) in particular for the model \texttt{g2}.
As underlined by one of the reviewers the sign of the bias when $\widetilde R \geq R$ is violated can be explained by the fact that $V$ is underestimated for attractive point processes and overestimated for repulsive point processes. The closer $\widetilde {R}$ is 
  to the finite range $R$, the lower the standard deviations are. The number of points with no $\widetilde R$-close neighbours is lower and lower as $\widetilde R$ 
 grows up, which explains this variance increase. Finally, we can also highlight that the asymptotic results are partly checked.
  When the window varies from $[0,1]^2$ to $[0,2]^2$, the standard deviations of the estimates are divided by 2 which turns out to be the square root 
 of the ratios of the domains areas.

 Table~\ref{tab:sim1} points out the dangers to choose a wrong value for the finite range parameter $R$. From a practical point of view, if this parameter is not set by the user, it has to be estimated. When the Papangelou conditional intensity is entirely modeled the parameters can be estimated using for instance the maximum likelihood or the pseudo-likelihood (e.g. \cite{B-MolWaa04}) and the irregular parameters (such as the finite range $R$) can be estimated using the profile pseudo-likelihood (\cite{A-BadTur00}). Such an approach can not be used here since we do not want to model the higher order interaction terms. We noticed that for a given (repulsive) point pattern our estimate is increasing with the value of $\widetilde R$ until a certain value where the regime is changing. Therefore we propose to estimate the finite range parameter $R$ as follows: a) Select a reasonable grid of values for $\widetilde R$ (we chose $[0.02,0.08]$ in our example). b)  Compute $\widehat \beta(\widetilde R)$ for each value of the grid. c) We consider a piecewise linear regression of $\widehat \beta(\widetilde R)$ in terms of $\widetilde R$ and we define $\widehat R$ as the breakpoint on the slope  (for this step we used the \texttt{R} package \texttt{segmented} \cite{muggeo:08}). Empirical results for the models \texttt{s1} and \texttt{s2} are presented in Table~\ref{tab:Rest}. Using the strategy described above, we manage to estimate $R$ quite well. The resulting estimates of the parameter $\beta^\star$ seem to be still unbiased. We note that the standard deviation is greater than in the case where $R$ is exaclty known (i.e. column $p=1$ in Table~\ref{tab:sim1}). However the loss of efficiency is quite reasonable.

 Finally, we investigate the finite sample properties of the estimator of the asymptotic variance of $\widehat \beta$ given by~\eqref{eq:sigmaEst} and the asymptotic normality result~\eqref{eq:clt2}. Table~\ref{tab:ic} focuses on the models \texttt{s1} and \texttt{s2} and displays the 95\% empirical coverage rates (i.e. the fraction of asymptotic confidence intervals constructed from the asymptotic normality result \eqref{eq:clt2} covering the true parameter value, here $\beta^\star=200$). As expected, the theoretical coverage rate is respected whatever the value of $\widetilde R\geq R$ and far from the result when $\widetilde R=90\% R$. Depite Proposition~\ref{prop:asymp} does not cover this case, it is interesting to note that the empirical results are still quite satisfactory if we replace $R$ by its estimated value.

 \begin{table}[H]
 {\small
 \begin{center}
 \begin{tabular}{rrrrrrr}
   \hline
   && \multicolumn{4}{c}{Mean (std dev.) for $\widetilde{R}=R \times p$} \\
 Model & $\overline n$ & $p=0.9$ & $p=1$ & $p=1.1$ & $p=1.2$ \\ 
   \hline
 \texttt{s1} $L=1$ & 99  & 174.0 (26.2) & 203.6 (34.2) & 204.3 (38.3) & 205.5 (44.4) \\ 
 $L=2$ & 393  & 172.6 (12.7) & 200.5 (16.4) & 200.3 (18.5) & 200.6 (20.8) \\ 
   \texttt{s2} $L=1$ & 156  & 192.0 (31.3) & 203.3 (41.8) & 203.9 (48.1) & 204.7 (55.8) \\ 
    $L=2$ & 622 & 190.7 (16.0) & 200.7 (19.5) & 200.3 (22.8) & 200.4 (26.5) \\ 
 \hline
   \texttt{shc1} $L=1$ & 94  & 174.3 (23.2) & 201.9 (30.3) & 203.3 (36.8) & 205.0 (41.9) \\ 
   $L=2$ & 379 &  174.1 (12.1) & 201.6 (15.7) & 201.9 (17.8) & 202.1 (21.1) \\ 
   \texttt{shc2} $L=1$ & 130 &  194.3 (30.2) & 205.4 (36.7) & 207.0 (42.8) & 209.1 (51.6) \\ 
   $L=2$ & 514 &  190.2 (15.2) & 198.5 (17.8) & 199.5 (21.1) & 200.6 (24.4) \\ 
 \hline
   \texttt{ps1} $L=1$ & 111  & 172.2 (25.1) & 202.6 (33.2) & 203.5 (38.7) & 206.5 (47.2) \\    
   $L=2$ & 445 & 171.5 (12.4) & 201.6 (16.3) & 201.7 (18.6) & 202.4 (21.6) \\ 
   \texttt{ps2} $L=1$ & 134 & 191.1 (30.7) & 201.7 (37.7) & 206.2 (45.2) & 208.8 (54.2) \\ 
    $L=2$ & 535&  192.8 (14.5) & 201.7 (17.5) & 202.1 (20.6) & 202.0 (24.4) \\ 
  \hline
   \texttt{t1} $L=1$ & 159 & 181.6 (34.6) & 204.5 (40.6) & 205.2 (48.7) & 205.8 (58.8) \\ 
  $L=2$ & 634 & 180.5 (16.7) & 201.6 (19.5) & 202.1 (22.9) & 202.8 (26.7) \\ 
   \texttt{t2} $L=1$ &178 &188.7 (30.7)& 200.7 (35.0) & 202.4 (39.9) & 205.1 (50.2)  \\ 
   $L=2$ & 727 &  189.3 (17.5) & 200.0 (20.2) & 199.8 (24.5) & 199.8 (30.8) \\ 
 \hline
   \texttt{g1} $L=1$ & 110 &  174.8 (26.6) & 201.4 (34.4) & 201.5 (40.2) & 203.6 (46.0) \\ 
   $L=2$ & 440 &  175.0 (13.3) & 200.7 (17.4) & 201.3 (19.9) & 201.6 (23.1) \\ 
  \texttt{g2} $L=1$ & 70 & 54.5 (10.4) & 50.0 (10.4) & 50.2 (11.4) & 50.5 (12.1) \\ 
   $L=2$ & 276 &  54.4 (5.4) & 50.2 (5.4) & 50.2 (5.8) & 50.3 (6.3) \\ \hline
 \end{tabular}
 \end{center}}
 \caption{\label{tab:sim1} Mean and standard deviations of estimates of the Poisson intensity parameter $\beta^{\star}$ for different parameters $\widetilde R$ based on 500 replications of different models generated on the window $[0,L]^2$ and estimated on the window $[\widetilde R,L-\widetilde R]^2$ for $L=1,2$. The second column, $\overline n$, indicates the average number of points (over the 500 replications) falling into $[0,L]^2$, i.e. the Monte-Carlo estimation of $\E n(\bX\cap[0,L]^2)$.}
 \end{table}

\begin{table}[H]
\centering
\begin{tabular}{rrr}
\hline
& \multicolumn{1}{c}{$\widehat R$} & \multicolumn{1}{c}{$\widehat \beta$} \\
& mean (std dev.) & mean (std dev.) \\
\hline
\texttt{s1} $L=1$ & 0.052 (0.004) &197.7 (38.8)\\
 $L=2$ & 0.051 (0.003) & 198.9 (22.1)\\
\texttt{s2} $L=1$ &0.051  (0.004) &201.7 (38.9)\\
 $L=2$ & 0.051 (0.003)& 198.4 (21.9)\\
 \hline
\end{tabular}
\caption{\label{tab:Rest} Mean and standard deviations of estimates of the finite range parameter and related Poisson intensity parameter estimates (i.e. $\widehat \beta(\widehat R)$) based on 500 replications of Strauss models on the window $[0,L]^2$.}
\end{table}

\begin{table}[H]
\centering
\begin{tabular}{rrrrrr}
\hline
& \multicolumn{4}{c}{$\widetilde R = p\times R$} & $\widehat R$\\
& $p=0.9$ & $p=1$ & $p=1.1$  &$p=1.2$ &\\
\hline
\texttt{s1} $L=1$ &  77.4\% & 94.8\% &95.0\% &94.2\%& 89.2\%\\
 $L=2$ & 40.0\%&94.8\%& 93.8\%&95.4\% & 90.4\%\\
\texttt{s2} $L=1$ &82.6\% & 93.2\%& 94.6\%& 93.8\% & 93\%\\
 $L=2$ & 69.8\%&95.0\%&93.6\%&92.8\% & 90.8\%\\
 \hline
\end{tabular}
\caption{\label{tab:ic}  Empirical coverage rates (i.e. the fraction of 95\% confidence intervals covering the true parameter value $\beta^\star=200$) based on 500 replications of Strauss models on the window $[0,L]^2$.}
\end{table}

\appendix
\section{Proof of Proposition~\ref{prop:asymp}} \label{sec-proofs}

We use the following additional notation in the proofs. The notation $\sum^{\neq}$ stands for summation over  distinct pairs of points, $\Lambda^c$ denotes the complementary set of $\Lambda$ in $\R^d$. 
Let $\tau= (\tau_x)_{x\in\R^d}$ be the shift group, where $\tau_x$ is the translation by the vector  $-x\in\R^d$. For brevity, when there is no ambiguity, we skip the dependence on $\bX$ and on $\widetilde R$ of the variables $N_{\LnR}, V_\LnR, W_\LnR, \widehat \sigma_n^2$.  

\begin{proof}

(i) There exists at least one stationary Gibbs 
measure. If this measure is unique, it is ergodic. Otherwise, it can be represented as a mixture
of ergodic measures (see \cite[Theorem 14.10]{A-Geo77}). Therefore, we can assume,
for this proof, that $P_{\beta^\star}$ is ergodic. We apply a general ergodic theorem for spatial point processes obtained
by \cite{A-NguZes79b}. For this, we have to check that $N_{\Lambda}$ and $V_\Lambda$ are additive, invariant by
 translation and satisfy $\E N_{\Lambda_0} <\infty$ and $\E V_{\Lambda_0} <\infty$, where for instance $\Lambda_0$ is the unit cube $[0,1]^d$. The additivity property is straightforward.
 Regarding the invariance by translation, let $\bx \in \Omega$, $\Lambda\subset \R^d$ and $y \in \R^d$, we have
\begin{align*}
N_{\tau_{y} \Lambda}(\tau_{y}\bx;\widetilde{R})
&= \displaystyle \sum_{u\in{\tau_{y}\bx}}\ind_{\tau_{y}\Lambda}(u)\ind(d(u, \tau_{y}\bx\setminus u)>\widetilde{R})      \\
&= \displaystyle \sum_{u\in{\tau_{y}\bx}}\ind_{\Lambda}(u+y)\ind(d(u, \tau_{y}\bx\setminus u)>\widetilde R)  \\
&= \displaystyle \sum_{v\in{\bx}}\ind_{\Lambda}(v)\ind(d(v-y, \tau_{y}\bx\setminus(v-y))>\widetilde R)  \\
&= \displaystyle \sum_{v\in{\bx}}\ind_{\Lambda}(v)\ind(d(v, \bx\setminus v)>\widetilde R) = N_{\Lambda}(\bx;\widetilde{R}).
\end{align*}
Following similar arguments, we also have $V_{\tau_y\Lambda}(\tau_y \bx;\widetilde{R}) =V_\Lambda(\bx;\widetilde{R})$. Let us now focus on the integrability conditions. First, we have
\[
\E \big|V_{\Lambda_0}\big| = \E V_{\Lambda_0}=
 \E \int_{\Lambda_0}\ind(\bX\cap B(u,\widetilde R)=\emptyset)\dd u \\
=|\Lambda_0|(1-F(\widetilde R)) =1-F(\widetilde R)< \infty.   
\]
Second, from Proposition~\ref{prop:h} 
\[
  \E \big|N_{\Lambda_0}\big|= \E N_{\Lambda_0} = \beta^\star \E V_{\Lambda_0} = \beta^\star(1-F(\widetilde R)).
\]
Therefore as $n\rightarrow \infty$, we derive the following almost sure convergences 
\begin{align}
|\LnR|^{-1}N_{\LnR} &\longrightarrow \E N_{\Lambda_0}=\beta^\star(1-F(\widetilde{R})) \nonumber\\
|\LnR|^{-1}V_{\LnR} &\longrightarrow \E V_{\Lambda_0}=1-F(\widetilde{R}). \label{eq:convVn}
\end{align}
Since the function $(y,z)\mapsto \frac{y}{z} $ is continuous for any $(y,z)$ in the open set
$\{(y,z)\in \R^2:z\neq 0\}$, we deduce the expected result  for any fixed $\widetilde R \geq R$.

(ii) From the definition of the innovations process \eqref{eq:defInn}, we have
\begin{equation} \label{eq:betaInn}
  |\LnR|^{1/2} ( \widehat \beta_n -\beta^\star) = |\LnR|^{1/2} \frac{ N_\LnR - \beta^\star V_\LnR}{V_\LnR} = 
\frac{ |\LnR|^{-1/2} I_\LnR(\bX,h)}{|\LnR|^{-1}V_\LnR}.
\end{equation}
Henceforth from Slutsky's Theorem and \eqref{eq:convVn}, the result will be deduced if we prove that $|\LnR|^{-1/2} I_\LnR(\bX,h)$ tends to a zero mean
 Gaussian distribution. This will be done by 
combining two papers \cite{A-CoeDerDroLav12} and \cite{A-CoeRub12} dealing with asymptotic normality and variance calculation of innovations processes. For brevity, we write $I_\Lambda$ in place of $I_\Lambda(\bX,h)$.

\bigskip

Following \cite[Lemma~3]{A-CoeDerDroLav12}, we have to assume that $P_{\beta^{\star}}$ is ergodic and to check that for any bounded domain $\Lambda \subset \R^d$: (I) $\E |I_{\Lambda}|^3<\infty$, (II) $\E I_{\Gamma_n}^2 \to 0$ for any sequence of bounded domains $\Gamma_n\subset \R^d$ such that $\Gamma_n\to 0$ as $n\to \infty$ and 
(III) $I_\Lambda$ depends on $X_{\Lambda\oplus \widetilde R}$. (III) follows from the finite range property \eqref{eq:range} and the definition of $h$ given by 
\eqref{eq:defh}. We now focus on verifying~(I). We have from Cauchy-Schwarz inequality
\begin{equation}\label{eq:I3}
  \E |I_\Lambda|^3 \leq 4 \left( \E  N_\Lambda^3 + {\beta^\star}^3\E  V_\Lambda^3 \right).
\end{equation}
On the one hand, we have since for any $m\geq 1$,  $h^m(u,\bx)=h(u,\bx)$
\begin{align}
\E  N_\Lambda^3 =& \E N_\Lambda + 3 \E \sum_{u,v \in X_\Lambda}^{\neq} h(u,\bX\setminus u)h(v,\bX\setminus v) \nonumber\\
&+ \E \sum_{u,v,w \in X_\Lambda}^{\neq} h(u,\bX\setminus u)h(v,\bX\setminus v) h(w,\bX\setminus w). \label{eq:T123} 
\end{align}
Let $T_1,T_2$ and $T_3$ denote the three terms of the right-hand side of~\eqref{eq:T123}. Following the proof of (i) in Proposition~\ref{prop:asymp}, $\E T_1\leq \beta^\star |\Lambda|$. Now, for any $u,v \in \R^d$ and $\bx \in \Omega$, we define the second order Papangelou conditional intensity by
\[
  \lambda_{\beta^\star} (u,v,\bx) = \lambda_{\beta^\star} (u,\bx)\lambda_{\beta^\star} (v,\bx \cup u) 
  =\lambda_{\beta^\star} (v,\bx)\lambda_{\beta^\star} (u,\bx \cup v).
\]
By using iterated versions of the GNZ formula 
\eqref{eq:gnz}, we derive the following computations
\begin{align*}
  T_2 &=3 \E \int_{\Lambda}\int_\Lambda h(u,\bX \cup v)h(v,\bX \cup u )\lambda_{\beta^\star}(u,v,\bX)\dd u\dd v \\ 
  &=  3{\beta^\star}^2\E \int_\Lambda\int_\Lambda \ind(d(u,\bX \cup v)> \widetilde R)\ind(d(v,\bX \cup u)> \widetilde R) \widetilde \lambda(u,\bX)\widetilde\lambda(v,\bX\cup u)\dd u\dd v \\
  &= 3{\beta^\star}^2 \E\int_\Lambda \int_\Lambda \ind(d(u,\bX \cup v)> \widetilde R)\ind(d(v,\bX \cup u)> \widetilde R) \widetilde\lambda^2(u,\emptyset) \dd u \dd v\\
   &\leq 3{\beta^\star}^2 |\Lambda|^2.
 \end{align*}
Using similar ideas, we leave the reader to check that $T_3 \leq {\beta^\star}^3 |\Lambda|^3$. On the other hand, we also check that
\begin{align*}
  \E  V_\Lambda^3 & = 
  \E\int_\Lambda\int_\Lambda\int_\Lambda \ind(d(u,\bX)> \widetilde R)\ind(d(v,\bX)> \widetilde R) \ind(d(w,\bX)> \widetilde R) \dd u \dd v \dd w \\
& \leq 
 |\Lambda|^3.
\end{align*}
We finally obtain
\[
  \E|I_\Lambda|^3 \leq 4\beta^\star|\Lambda|+12{\beta^{\star}}^2|\Lambda|^2+8{\beta^{\star}}^3|\Lambda|^3 < \infty
\]
whereby (I) is checked. Following the proof of (I), we may prove that for any bounded domain $\Lambda$
\[
  \E I_\Lambda^2 \leq 2\beta^\star |\Lambda| + 4{\beta^\star}^2 |\Lambda|^2
\]
which proves (II). Therefore we can invoke \cite[Lemma~3]{A-CoeDerDroLav12} to obtain the asymptotic normality of $\widehat \beta_n$. 

The derivation of the asymptotic 
variance of $|\LnR|^{-1}\Var I_{\LnR}(\bX,h)$ is done using Lemma~\ref{lem:varI} in Appendix~\ref{sec:appVar} with the function $g=h$ given 
by~\eqref{eq:defh}. We omit the details and leave the reader to check that $h$ fulfills the integrability assumptions in Lemma~\ref{lem:varI} and focus
 on the derivations of the quantities $A_1,A_2$ and $A_3$. First, it is clear that $A_1=\beta^\star (1-F(\widetilde R))$. Second, note that from~\eqref{eq:range} we have for any $v\in B^c(0,R)$ and $x\in \Omega$
 \[
   \lambda_{\beta^\star}(0,\bx)\lambda_{\beta^\star}(v,\bx)-\lambda_{\beta^\star}(0,v ,\bx)=0.
 \]
Since $B(0,R)\subseteq B(0,\widetilde R)$  
\[
  \E \int_{B(0,\widetilde R)\backslash B(0,R)}\big(\lambda_{\beta^\star}(0,\bX)\lambda_{\beta^\star}(v,\bX)-\lambda_{\beta^\star}(0,v ,\bX)\big)dv=0
\]
which means that we can substitute $R$ by $\widetilde R$ in the term $A_2$. Then, we derive
\begin{align*}
A_2=& \E \int_{ B(0,\widetilde R)} \ind(d(0,\bX)> \widetilde R)\ind(d(v,\bX)>\widetilde R)\big({\beta^\star}^2-\lambda_{\beta^\star}(0,v,\bX) \big) \dd v. 
\end{align*}
Third,
\begin{align*}
A_3=& \E \int_{{B}(0,\widetilde R)} \big( \ind(d(0,\bX \cup v)> \widetilde R) -\ind(d(0,\bX)> \widetilde R) \big) \\
&\quad \times
\big( \ind(d(v,\bX \cup 0)> \widetilde R) -\ind(d(v,\bX)> \widetilde R) \big) 
\lambda_{\beta^\star}(0,v,\bX) \dd v \\
=& \E \int_{ B(0,\widetilde R)} \ind(d(0,\bX)> \widetilde R)\ind(d(v,\bX)> \widetilde R)\lambda_{\beta^\star}(0,v,\bX)  \dd v \\
\end{align*}
which leads to 
\begin{align*}
  A_2+A_3 &= {\beta^\star}^2 \E \int_{ B(0,\widetilde R)} \ind(d(0,\bX)> \widetilde R)\ind(d(v,\bX)> \widetilde R)\dd v\\
  &= {\beta^\star}^2 \int_{ B(0,\widetilde R)} (1- F_{0,v}(\widetilde R))\dd v.
\end{align*}
In other words
\[
  |\LnR|^{-1} \Var I_\LnR (\bX,h) \rightarrow \beta^\star (1-F(\widetilde R))+
   {\beta^\star}^2 \int_{ B(0,\widetilde R)} (1- F_{0,v}(\widetilde R))\dd v
\]
which leads to the result from \eqref{eq:convVn} and \eqref{eq:betaInn}.

(iii) We check that the proposed estimate $\widehat \sigma_n^2$ tends almost surely towards $\sigma^2$ as $n\to \infty$
 (a convergence in probability would be sufficient). From \eqref{eq:sigmaEst}, \eqref{eq:convVn} and \eqref{eq:sigma}, we need to prove that
\[
  |\LnR|^{-1} W_{\LnR} \longrightarrow \int_{{B}(0,\widetilde R)}  (1-F_{0,v}(\widetilde R)) \dd v.
\] 
We decompose $W_{\LnR}$ as $W_{\LnR}^{(1)}+W_{\LnR}^{(2)}$ where
\begin{align*}
W_{\LnR}^{(1)} &= \int_{\LnR_{\ominus \widetilde R}} \int_{ B(u,\widetilde R)} \ind(d(u,\bX)> \widetilde R)\ind(d(v,\bX)> \widetilde R) \dd u \dd v \\
W_{\LnR}^{(2)}&= \int_{\LnR \setminus \LnR_{\ominus \widetilde R} } \int_{ B(u,\widetilde R)\cap \LnR} \ind(d(u,\bX)> \widetilde R)\ind(d(v,\bX)> \widetilde R) \dd u \dd v 
\end{align*}
where $\LnR_{\ominus \widetilde R}=\{u\in \LnR: B(u,\widetilde R)\subseteq \LnR \}$.
Regarding the first term, we can now apply the ergodic theorem to prove that almost surely
\begin{align*}
|\LnR_{\ominus \widetilde R}|^{-1} W_{\LnR}^{(1)} &\longrightarrow \E \int_{{B} (0,\widetilde R)}  \ind(d(0,\bX)> \widetilde R)\ind(d(v,\bX)> \widetilde R) \dd v 
\end{align*}
where the latter comes from the finite range assumption \eqref{eq:range} and the fact that $\widetilde R\geq R$.
Since $|\LnR| \sim |\LnR_{\ominus \widetilde R}|$ as $n\rightarrow \infty$, the proof is done by verifying that $|\LnR|^{-1} W_{\LnR}^{(2)} \to 0$ which is straightforward since 
\[
  |\LnR|^{-1} W_{\LnR}^{(2)} \leq \frac{|\LnR \setminus \LnR_{\ominus \widetilde R}|}{|\LnR|} \; | B(u,\widetilde R)| = o(1).
\]
\end{proof}
\section{Auxiliary lemma}\label{sec:appVar}

The following result is a particular case of \cite[Proposition~3.2]{A-CoeRub12}.
\begin{lemma} \label{lem:varI}
Assume the Papangelou conditional intensity of the point process $\bX$ satisfies \eqref{eq:lbeta} and \eqref{eq:range} with finite range $R$.
Let $g:\R^d\times \Omega \to \R$ be a nonnegative measurable function satisfying for any $u\in \R^d$ and  $\bx \in \Omega$
\[
 g(u,\bx)=  g(0,\tau_{u}\bx)   \quad \mbox{ and } \quad g(u,\bx)=g(u,\bx_{ B(u,\widetilde R)})
\]
with $\widetilde R\geq R$. We also assume that the following integrability conditions are fulfilled 

\begin{align*}
&\E [ g^2(0,\bX) \lambda_{\beta^\star}(0,\bX)] <\infty \\
&\E  \int_{ B(0,R)} g(0,\bX)g(v,\bX) \big|\lambda_{\beta^\star}(0,\bX)\lambda_{\beta^\star}(v,\bX)-
\lambda_{\beta^\star}(0,v ,\bX)\big| \dd v  <\infty\\
 & \E \int_{ B(0,\widetilde R)} 
\big|\big(g(0,\bX \cup v) -g(0,\bX)\big) \big( g(v,\bX \cup 0) -g(v,\bX)\big)\bigg
|
  \lambda_{\beta^\star}(0,v ,\bX) \dd v <\infty.
\end{align*}
Then, as $n \to \infty$
\[
  |\Lambda_n|^{-1}\Var I_{\Lambda_n}(\bX,g)  \longrightarrow A_1+A_2+A_3
\]
where
\begin{align*}
A_1 &= \E  [g^2(0,\bX) \lambda_{\beta^\star}(0,\bX)] \\
A_2 &= \E \int_{ B(0,R)} g(0,\bX)g(v,\bX) \big(  \lambda_{\beta^\star}(0,\bX)\lambda_{\beta^\star}(v,\bX)-
\lambda_{\beta^\star}(0,v,\bX)
 \big) \dd v \\
 A_3&= \E \int_{ B(0,\widetilde R)}
\big(g(0,\bX \cup v) -g(0,\bX)\big) \big( g(v,\bX \cup 0) -g(v,\bX)\big)
 \lambda_{\beta^\star}(0,v ,\bX) \dd v. 
 \end{align*}
\end{lemma}

\bibliographystyle{plain}

\bibliography{PoissInt}
 
\end{document}